\documentclass[11pt,a4paper]{amsart}
\usepackage[latin1]{inputenc}
\usepackage{amsxtra}
\usepackage{amssymb}
\usepackage{amsthm}
\usepackage{euler}
\usepackage{graphicx}
\usepackage[all,2cell,dvips]{xy}
\evensidemargin=\oddsidemargin
\makeatletter
\def\cleardoublepage{\clearpage\if@twoside \ifodd\c@page\else
  \hbox{}
  \thispagestyle{empty}
  \newpage
  \if@twocolumn\hbox{}\newpage\fi\fi\fi}
\makeatother
\usepackage{fancyhdr}
\headwidth=\textwidth
\theoremstyle{plain} 
\newtheorem{theorem}[subsection]{Theorem}
\newtheorem*{teo}{Theorem}
\newtheorem{corollary}[subsection]{Corollary}
\newtheorem{proposition}[subsection]{Proposition}
\newtheorem*{propo}{Proposition}
\newtheorem{lemma}[subsection]{Lemma}

\theoremstyle{remark}
\newtheorem*{remark}{Remark}

\theoremstyle{definition}
\newtheorem*{definition}{Definition}
\newtheorem*{map}{Map}
\newtheorem*{notation}{Notation}
\newtheorem*{algorithm}{Algorithm}

\newcommand{\diamantei}{\underset{i}{\diamond}}
\newcommand{\diamanteuno}{\underset{1}{\diamond}}

\newcommand{\circulito}{\underset{i}{\circ}}
\newcommand{\circR}{\underset{R}{\circ}}

\newcommand{\tensorE}{\underset{{\tiny\hbox{$E$}}}{\otimes}}

\newcommand{\tensork}{\underset{{\tiny\hbox{$k$}}}{\otimes}}

\newcommand{\cnk}{\left( \begin{array}{c} n+1 \\ l  \end{array}\right)}
\newcommand{\cnnk}{\left( \begin{array}{c} n-1 \\ l \end{array}\right)}
\newcommand{\cck}{\left( \begin{array}{c} n \\ l  \end{array}\right)}
\newcommand{\matrizDcero}{\scalebox{0.5}
{$\begin{pmatrix} 0&0\\D_0&0 \end{pmatrix} $}}
\newcommand{\matrizDuno}{\scalebox{0.5}
{$\begin{pmatrix} 0&0\\ D_1&0\end{pmatrix}$}}
\newcommand{\matrizDn}{\scalebox{0.5}
{$\begin{pmatrix} 0&0\\ D_n&0 \end{pmatrix}$}}
\newcommand{\pnl}{
\left( \begin{array}{c} n \\ l \end{array}\right)
-
\left( \begin{array}{c} n \\ l-1  \end{array}\right)
}
\newcommand{\qnk}{
\left( \begin{array}{c} n+1 \\ l \end{array}\right)
-
\left( \begin{array}{c} n+1 \\ l-1  \end{array}\right)
-
\left( \begin{array}{c} n-1 \\ l-1  \end{array}\right)
+
\left( \begin{array}{c} n-1 \\ l-2  \end{array}\right)
}
\newcommand{\mnz}{\left( \begin{array}{c} n-1 \\ l
    \end{array}\right)}
\title{The Lie module structure on the Hochschild cohomology groups of
  monomial algebras with radical square zero}
\author{Selene S{\'a}nchez-Flores}
\address{
Institut de Math{\'e}matiques et de Mod{\'e}lisation de Montpellier \\
Universit{\'e} de Montpellier II \\ 
Place Eug{\`e}ne Bataillon\\
F-34095 Montpe${}$llier Cedex 5\\ 
France.}
\email{sanchez@math.univ-montp2.fr}                       
\begin{document}
\maketitle
%%%%%%%%%%%%%%%%%%%%%%%%%%%%%%%%%%%%%%%%%%%%%%%%%%%%%%%%%%%%%%%%%%%%%%%%%
%%%%%%%%%%%%%%%%%%%%%%% Abstract %%%%%%%%%%%%%%%%%%%%%%%%%%%%%%%%%%%%%%%%

\begin{abstract}
We study the Lie module structure given by the Gerstenhaber bracket 
on the Hochschild cohomology groups of a monomial algebra 
with ra${}$dical square zero. The description of such Lie module 
structure will be given in terms of the combinatorics of the quiver. 
The Lie module structure will be related 
to the classification of finite dimensional
modules over simple Lie algebras when the quiver is given by the two
loops and the ground field is the complex numbers.
\end{abstract}

%%%%%%%%%%%%%%%%%%%%%%%%%%%%%%%%%%%%%%%%%%%%%%%%%%%%%%%%%%%%%%%%%%%%%%%%
%%%%%%%%%%%%%%%%%%%%%% Introduction %%%%%%%%%%%%%%%%%%%%%%%%%%%%%%%%%%%%

\section*{Introduction.}
Let $A$ be an associative unital $k$-algebra where $k$ is a field.
The {\it{$n^{th}$ Hochschild cohomology group of $A$ }}, 
denoted by $HH^n(A)$, refers to 
\[
HH^n(A):= HH^n(A,A)=\, Ext_{A^e}^n(A,A)
\]
where $A^e$ is the enveloping algebra $A^{op} \otimes_k A$ of $A$. 
Thus, for example, $HH^0(A)$ is the center of $A$ 
and the first Hochschild cohomology group $HH^1(A)$
is the vector space of the outer derivations. 
Note that the first Hochschild cohomology group has a 
Lie algebra structure given by the commutator bracket. 
In \cite{gersten}, Gerstenhaber introduced 
two operations on the Hochschild cohomology groups: 
the cup product and the bracket 
\[
[\, - \, , \, - \,]: 
HH^n(A) \times HH^m(A) \longrightarrow HH^{n+m-1}(A) .
\] 
He proved that the {\it{Hochschild cohomology of $A$ }}, 
$$
HH^{*}(A) := \displaystyle{\bigoplus_{n=0}^\infty \,  HH^{n}(A)} \, ,
$$
provided with the cup product is a graded commutative algebra.
Furthermore, he demonstrated that $HH^{*+1}(A)$ 
endowed with the Gerstenhaber bracket 
has a graded Lie algebra structure. 
Consequently, $HH^1(A)$ is a Lie algebra and  
$HH^n(A)$ is a Lie module over $HH^1(A)$. 
As a matter of fact, the Gerstenhaber bracket restricted to $HH^1(A)$ 
is the commutator Lie bracket of the outer derivations. 
Moreover, the cup product and the Gerstenhaber bracket endow 
$HH^*(A)$ with the so-called Gerstenhaber algebra structure. 

Besides, it was shown that the algebra structure on $HH^*(A)$ 
is invariant under derived equivalence \cite{happel, rickard}.
In addition, in \cite{keller}, Keller proved that the Gerstenhaber
bracket on $HH^{*+1}(A)$ is preserved under derived equivalence. 
Therefore, the Lie module structure on $HH^n(A)$ over $HH^1(A)$ is
also an invariant under derived equivalence.

Understanding both the graded commuative algebra and 
the graded Lie algebra structure,  
on the Hochschild cohomology of algebras 
is a difficult assigment. 
Different techniques have been used in order to:
$(1)$ describe the Hochschild cohomology algebra (or ring) 
for some algebras, 
\cite{holm, cibilssolotar,cibils, erdmannsnashall, erdmannholm, 
siegelwitherspoon, mariano, erdmannholmsnashall,cagliero,eu,xu}; 
$(2)$ study the Hochschild cohomology ring modulo nilpotence,
\cite{greensnashallsolberg,greensnashallsolberg2,greensnashall} and 
$(3)$ compute the Gerstenhaber bracket
\cite{bustamante,eu2}.

On the other hand, C. Strametz studied, in \cite{strametz}, 
the Lie algebra structure on the first Hochschild cohomology group 
of monomial algebras. 
She accomplishes to describe such Lie algebra structure in terms of 
the combinatorics of the monomial algebras. 
Moreover, she relates such description to the
algebraic groups which appear in Guil-Asensio and Saor{\'\i}n's study of
the outer automorphisms \cite{saorin}. 
In \cite{strametz}, Strametz also gave criteria for simplicity 
of the first Hochschild cohomology group.  

In this paper we are interested in the Lie module structure 
on the Hochschild cohomology groups induced by the Gerstenhaber bracket.
This approach was suggested by C. Kassel 
and motivated by the work of C. Strametz. 
The aim of this paper is to describe the Lie module structure 
on the Hochschild cohomology groups  
for monomial algebras of (Jacobson) radical square zero. 
Recall that a {\it{monomial algebra of radical square zero}} 
is the quotient of the path algebra of a quiver $Q$ 
by the two-sided ideal generated by the set of paths of length two.
We will use the combinatorics of the quiver 
in order to describe the Lie module structure. 
Moreover, for the case of the two loops quiver,  
we relate such Lie module structure of $HH^n(A)$ 
to the classification of the (finite-dimensional) irreducible Lie modules 
over $sl_2$ when the ground field is the complex numbers. 

The Hochschild cohomology groups of those algebras 
have been described in \cite{cibils} 
using the combinatorics of the quiver. 
Such description enables to prove 
that the cup product of elements of 
positive degree is zero when 
$Q$ is not an oriented cycle. 
In this paper, we use Cibils' description of $HH^n(A)$  
in order to study the Lie module structure on 
the Hochschild cohomology groups. 
First, we reformulate the Gerstenhaber bracket 
for the realization of the Hochschild 
cohomology groups obtained 
through the computations in \cite{cibils}. 
In the first section 
we construct two quasi-ismorphisms between the Hochschild cochain
complex and the complex induced by the reduced projective resolution. 
Then in the seco${}$nd section, 
using such quasi-isomorphisms, we introduce a new bracket; 
which coincides with the Gerstenhaber bracket. In the third
section, we use the combinatorics of the quiver to describe the
Gerstenhaber bracket.

In the last section, 
we study a particular case: the monomial algebra of
radical square zero given by the two loops quiver. 
For this algebra, we prove that  
$HH^1(A)$ is isomorphic as a Lie algebra 
to $gl_2 \mathbb{C}$ and then we identify a copy of 
$sl_2 \mathbb{C}$ in $HH^1(A)$.
In order
to decribe $HH^n(A)$ as a Lie module over $HH^1(A)$, we start studying 
the Lie module structure of $HH^n(A)$ over $sl_2 \mathbb{C}$. 
In this article, we determine the decomposition 
of $HH^n(A)$ into direct sum of
irreducible modules over $sl_2 \mathbb{C}$. 
Moreover, we show that 
such decomposition can be
obtained by an algorithm. 
In the following table we illustrate the decomposition for the
Hochschild cohomology groups of degrees between 2 and 7. 
We denote by $V(i)$ the unique irreducible Lie module of dimension
$i+1$ over $sl_{2} \mathbb{C}$.

\[
\scalebox{0.9}
{$
\begin{array}{c||ccccccccc}
n & 
V(0)&V(1)&V(2)&V(3)&V(4)&V(5)&V(6)&V(7)&V(8) \\
\hline \hline 
\, & 
\,  & \, & \, & \, & \, & \, & \, & \, & \, \\
HH^2(A) & 
\,  & 1  & \, & 1  & \, & \, & \, & \, & \, \\ 
\, & \,  & \,  & \, & \,  & \, & \, & \, & \, & \, \\
HH^3(A) & 
1   & \, & 2  & \,  & 1 & \, & \, & \, & \,  \\ 
\, & 
\,  & \,  & \, & \,  & \, & \, & \, & \, & \,\\
HH^4(A) & 
\,  & 3  & \, & 3  & \, & 1 & \, & \, & \,  \\ 
\, & 
\,  & \,  & \, & \,  & \, & \, & \, & \,& \,  \\
HH^5(A) & 
3  &  \,  & 6 & \, & 4  & \, & 1 & \, & \,  \\ 
\, & 
\,  & \,  & \, & \,  & \, & \, & \, & \, & \, \\
HH^6(A) & 
\,  & 9  & \, & 10 & \, & 5  & \, & 1  & \, \\
\, & 
\,  & \,  & \, & \,  & \, & \, & \, & \,& \,  \\
HH^7(A) & 
9  &  \,  & 19 & \, & 15  & \, & 6 & \, & 1  \\ 
\end{array}
$} 
\]
$\quad$  
\\  

In the above table, let us remark that the three 
last diagonal form a component of the Pascal triangle. Note also
that the integer sequence given by the first and second column are the
same. We will prove that these two remarks are in general true.  
This will enable to show the validity of the algorithm and in
consequence obtain the other diagonals of the table.  Moreover, we
have introduced the sequence of numbers in the 
Encyclopedia of Integer Sequences 
[http://www.research.att.com/~njas/sequences/index.html], 
it appears to be related with two sequences. 
Among these sequence, there is one that represents
the expected saturation of a binary search tree (or BST) on n nodes
times the number of binary search trees on n nodes, or alternatively,
the sum of the saturation of all binary search trees on n
nodes. Another sequence gives the number of standard tableaux of
shapes (n+1,n-1). The two sequences are given by explicit
formulas. 

In a future paper, 
we will apply the same techniques, 
as those we use in this article, 
to prove that the first Hochschild cohomology group 
of the monomial algebra of radical square zero is   
the Lie algebra $gl_{n} \mathbb{C}$ when the quiver 
is given by $n$ loops.  
Moreover, we will determine, as we did for the two loops case,  
the decomposition into direct sum of irreducible 
modules over $sl_{n} \mathbb{C}$ 
but only for the second Hochschild cohomology group.  
We will also be dealing 
with the case when the quiver has no loops and no cycles.

%%%%%%%%%%%%%%%%%%%%%%%%%%%%%%%%%%%%%%%%%%%%%%%%%%%%%%%%%%%%%%%%%%%%%%%%
%%%%%%%%%%%%%%%%%%%%% Thanks %%%%%%%%%%%%%%%%%%%%%%%%%%%%%%%%%%%%%%%%%%%

\thanks{{\bf{Acknowledgment.}} This work will be part of my PhD thesis at
the University of Montpellier 2. I am indebted to my
advisor, Professor Claude Cibils, not only for valuable discussions
about the subject and his helpful remarks on this paper, 
but also for his encouragement. I would like to thank the referee for
helpful suggestions in improving this paper. }

%%%%%%%%%%%%%%%%%%%%%%%%%%%%%%%%%%%%%%%%%%%%%%%%%%%%%%%%%%%%%%%%%%%%%%%
%%%%%%%%%%%%%%%%%%%% First Section %%%%%%%%%%%%%%%%%%%%%%%%%%%%%%%%%%%%

\section{ A comparison map beetween 
the bar projective resolution and 
the reduced bar projective resolution.}
In this section, we deal with finite dimensional $k$-algebras 
whose semisimple part (i.e the quotient by its radical) is
isomorphic to a finite number of copies of the field. Monomial
algebras of radical square are a particular case of these algebras. 
\subsection*{Two projective resolutions.}
The usual $A^e$-projective resolution of $A$
used to calculate the Hochschild cohomology groups 
is the standard bar resolution.
The {\it{standard bar resolution}}, that we will denote by $\bold{S}$, 
is given by the following exact sequence: 
\[
\bold{S}:= \qquad \quad 
\cdots 
\rightarrow
A^{\otimes^{n+1}_k}
\stackrel{\delta}{\rightarrow} \, 
A^{\otimes^{n}_k} 
\stackrel{\delta}{\rightarrow} \, 
\cdots  
\stackrel{\delta}{\rightarrow} \,
A^{\otimes^3_k} 
\stackrel{\delta}{\rightarrow} \, 
A \tensork A  
\stackrel{\mu}{\rightarrow}    \,
A  
\rightarrow                    \, 
0
\]
where
$\mu$ is the multplication and 
the $A^e$-morphisms $\delta$ are given by
\[
\delta (x_1 \otimes \dots \otimes x_{n+1})= 
\sum_{i=1}^{n} \, (-1)^{i+1} 
x_1 \otimes \dots \otimes x_ix_{i+1} \otimes \dots \otimes x_{n+1}
\] 
where $x_i \in A$ and $\otimes$ means $\tensork$. 

Now, the $A^e$-projective resolution of $A$ used in \cite{cibils} 
to compute the Hochschild cohomology groups 
of a monomial radical square zero is 
the {\it{reduced bar resolution}}. 
It is 
defined for a finite dimensional $k$-algebra $A$ whose  
Wedderburn-Malcev decomposition is given by the direct sum  
$A=E \oplus r$ where $r$ is the Jacobson radical of $A$ and 
$E \cong A/r \cong k \times k \cdots \times k$. 
In the sequel $A$ denotes an algebra verifying those conditions.  
Let us denote by $\bold{R}$ the reduced bar resolution. 
It is given by the following exact sequence:
\[
\bold{R}:=
\cdots 
\rightarrow 
A \tensorE r^{\otimes^{n+1}_E} \tensorE A  
\stackrel{\delta}{\rightarrow} \, 
A \tensorE  r^{\otimes^{n}_E} \tensorE A 
\stackrel{\delta}{\rightarrow} \, 
\cdots 
\stackrel{\delta}{\rightarrow} \,
A \tensorE r \tensorE A 
\stackrel{\delta}{\rightarrow} \, 
A \tensorE A  
\stackrel{\mu}{\rightarrow}    \, 
A  
\rightarrow                    \, 0
\]
where $\mu$ is the multplication and 
the $A^e$-morphisms $\delta$ are given by
\[
\begin{array}{cl}
\delta (a \otimes x_1 \otimes \dots \otimes x_{n+1} \otimes b) & 
= ax_1 \otimes x_2 \otimes \dots \otimes x_{n+1} \otimes b \\
\, & + \, \sum_{i=1}^{n} \, (-1)^i 
a \otimes x_1 \otimes \dots \otimes x_ix_{i+1} \otimes \dots \otimes b \\
\, & + \, (-1)^{n+1} \, 
a \otimes x_1 \otimes \dots \otimes x_{n} \otimes x_{n+1}b
\end{array}
\] 
where $a,b \in A$, $x_i \in r$ and $\otimes$ means $\tensorE$. 
The proof that this sequence is a projective resolution 
can be found in \cite{cibils2}. 

\subsection*{Comparison maps.}
Theorically, a comparison map exists 
between these two projective resolutions.  
The objective of this section 
is to give an explicit comparison map
between the projective resolutions 
$\bold{S}$ and $\bold{R}$ in both directions. 
Such comparison map will induce 
some quasi-isomorphisms
between the Hochschild cochain complex
and the complex induced by the
reduced bar resolution. 
The explicit calculations 
of these quasi-isomorphisms, 
enables to 
reformulate 
the Gerstenhaber bracket.

In this paragraph, we are going to give two maps of complexes: 
$$
p:\bold{S} \rightarrow \bold{R} \text{ and } 
s:\bold{R} \rightarrow \bold{S}.$$ 
This means we will define maps $(p_n)$ and $(s_n)$ 
such that the next diagram 
\begin{equation}\label{Diagrama}
\xymatrix{
\: \cdots \: A \tensork A^{\otimes^{n+1}_k} \tensork A 
\ar[d]_{p_{n+1}} \ar[r]^\delta & 
A \tensork A^{\otimes^n_k} \tensork A \ar[d]_{p_n} \: \cdots &
A \tensork A \ar[d]_{p_0} \ar[r]^{\mu} & A \ar[d]_{id} \ar[r] & 0 \\
\: \cdots \: A \tensorE r^{\otimes^{n+1}_E} \tensorE A  
\ar[r]^\delta \ar[d]_{s_{n+1}} &
A \tensorE r^{\otimes^n_E} \tensorE A \ar[d]_{s_n} \: \cdots &
A \tensorE A \ar[d]_{s_0} \ar[r]^{\mu}  & A \ar[d]_{id} \ar[r] & 0 \\
\: \cdots \: A \tensork A^{\otimes^{n+1}_k} \tensork A \ar[r]^{\delta} & 
A \tensork A^{\otimes^n_k} \tensork A \: \cdots & 
A \tensork A \ar[r]^\mu & A \ar[r] & 0
}
\end{equation}
commutes. 
\begin{map}[$p_n$]
We define $p_0$ as the linear map given by 
\[
\begin{array}{rlllc}
  p_0: & A \tensork A & \rightarrow  & A \tensorE A & \, \\ 
  \,   & a \tensork b & \mapsto      & a \tensorE b & .
\end{array}
\]
Now, let $n \geq 1$. Define 
\[
p_n: A \tensork A^{\otimes^n_k} \tensork A \rightarrow 
     A \tensorE r^{\otimes^n_E} \tensorE A
\]
as the linear map given by 
\begin{center}
\scalebox{0.9}{$
a \tensork     
x_1 \tensork \cdots \tensork x_i  \tensork \cdots \tensork   x_{n+1} 
\tensork b 
\mapsto
a \tensorE 
\pi(x_1)\tensorE \cdots \tensorE \pi(x_i) \tensorE \cdots \tensorE \pi(x_{n+1})
\tensorE b 
$.}
\end{center}
where $\pi$ denotes the projection map from $A$ to the Jacobson
radical square zero.
Notice that $p_n$ is an $A^e$-morphism for all $n$.  
\end{map}
In order to define the maps $(s_n)$ we introduce some notation.  
In the sequel, let $E_0$ denote a  
complete system of idempotents and orthogonal elements of $E$. 
Note that the set $E_0$ is finite. 
\begin{remark}
Now, consider elements of 
$A \tensorE r^{\otimes^n_E} \tensorE A$ of the form
$$
ae_{j_1} \tensorE \cdots \tensorE e_{j_{i-1}} x_{i-1} e_{j_i} \tensorE
e_{j_i}x_ie_{j_{i+1}} \tensorE e_{j_{i+1}} x_{i+1} e_{j_{i+2}}
\tensorE \cdots \tensorE 
e_{j_{n+1}}b
$$
where each $e_{j_i}$ is in $E_0$, $a,b$ are in $A$ and 
$x_i$ in $r$. 
It is not difficult to see that 
those elements generate the vector space 
$A \tensorE r^{\otimes^n_E} \tensorE A$. Indeed, we have that 
\[
\begin{array}{l}
a \tensorE x_1 \tensorE \cdots \tensorE x_i \tensorE \cdots \tensorE 
x_{n} \tensorE b = \\
\, \\
\displaystyle{\sum_{ j_{1},\dots,j_{n+1} }}
ae_{j_1} \tensorE \cdots \tensorE e_{j_{i-1}} x_{i-1} e_{j_i} \tensorE
e_{j_i}x_ie_{j_{i+1}} \tensorE e_{j_{i+1}} x_{i+1} e_{j_{i+2}} \tensorE
\cdots \tensorE 
e_{j_{n+1}}b 
\end{array}
\]
where the sum is over all $(n+1)$-tuples 
$(e_{j_1},\dots,e_{j_i},\dots,e_{j_{n+1}})$ of elements of $E_0$.
\end{remark}

\begin{map}[$s_n$]
Define $s_0$ as the linear map given by 
\[
\begin{array}{rlllc}
  s_0: & A \tensorE A & \rightarrow  & A \tensork A   & \, \\ 
  \,   & ae \tensorE eb & \mapsto    & ae \tensork eb & .
\end{array}
\]
So we have that
\[
s_0(a \tensorE b)= \displaystyle{\sum_{e \in E_0} ae \tensork eb} \, .
\]
It is well defined because 
$s_0(ae \tensorE b)= ae \tensork eb = s_0(a\tensorE eb)$ for
all $e \in E$. Now, let  $n \geq 1$. Define 
\[
s_n: A \tensorE r^{\otimes^n_E} \tensorE A \rightarrow
     A \tensork A^{\otimes^n_k} \tensork A 
\]  
as the linear map given by 
\[
\begin{array}{l}
ae_{j_1} \tensorE \cdots \tensorE e_{j_{i-1}} x_{i-1} e_{j_i} \tensorE
e_{j_i}x_ie_{j_{i+1}} \tensorE e_{j_{i+1}} x_{i+1} e_{j_{i+2}} \tensorE
\cdots \tensorE 
e_{j_{n+1}}b
\mapsto \\
ae_{j_1} \tensork \cdots \tensork e_{j_{i-1}} x_{i-1} e_{j_i} \tensork
e_{j_i}x_ie_{j_{i+1}} \tensork e_{j_{i+1}} x_{i+1} e_{j_{i+2}} \tensork
\cdots \tensork
e_{j_{n+1}}b 
\end{array}
\]
where each $e_{j_i}$ is in $E_0$. So we have that
\[ 
\begin{array}{l}
s_{n} (a \tensorE x_1 \tensorE \cdots \tensorE x_i \tensorE \cdots \tensorE 
x_{n} \tensorE b )= \\
\, \\
\displaystyle{\sum_{j_{1},\dots,j_{n+1} }}
ae_{j_1} \tensork \cdots \tensork e_{j_{i-1}} x_{i-1} e_{j_i} \tensork
e_{j_i}x_ie_{j_{i+1}} \tensork e_{j_{i+1}} x_{i+1} e_{j_{i+2}} \tensork
\cdots \tensork 
e_{j_{n+1}}b 
\end{array}
\]
where the sum is over all $(n+1)$-tuples 
$(e_{j_1},\dots,e_{j_i},\dots,e_{j_{n+1}})$ of elements of $E_0$. 
Notice that $s_n$ is an $A^e$-morphism. 
\end{map}
\begin{remark}
It is clear that $p_ns_n=id_{A \tensorE r^{\otimes_E^n} \tensorE A}.$
\end{remark}
\begin{lemma}\label{map}
The maps 
$$
p:\bold{S} \rightarrow \bold{R} \text{ and } 
s:\bold{R} \rightarrow \bold{S}$$ 
defined above are maps of complexes. 
\end{lemma}
\begin{proof} A straightforward verification shows that 
the diagram (\ref{Diagrama}) is commutative.
\end{proof}

\subsection*{Two complexes.}
We will denote the {\it{Hochschild cochain complex}} 
by $\bold{C^\bullet(A,A)}$. 
Recall that it is defined by the complex,  
\[
\begin{array}{rll}
0  \rightarrow  
A \stackrel{\delta}{\rightarrow} 
Hom_{k}(A,A) \stackrel{\delta}{\longrightarrow} & \cdots &\, \\
\cdots \: \longrightarrow 
& Hom_{k}(A^{\otimes^n_k},A) \stackrel{\delta}{\longrightarrow} 
Hom_{k}(A^{\otimes^{n+1}_k},A) & \cdots
\end{array}
\]
where 
$
\delta(a)(x)=xa-ax
$
for $a$ in $A$ and 
\[
\begin{array}{cl}
\delta f(x_1 \otimes \cdots \otimes x_n \otimes x_{n+1}) = & 
x_1f(x_2 \otimes \cdots \otimes x_{n+1}) \, + \\
\, & \sum_{i=1}^{n} (-1)^i f(x_1 \otimes \cdots \otimes x_ix_{i+1} \otimes
\cdots \otimes x_{n+1})+ \\
\, & (-1)^{n+1} f(x_1 \otimes \cdots \otimes x_n)x_{n+1}
\end{array}
\]
for $f$ in $Hom_{k}(A^{\otimes^{n}_k},A)$.
Notice that after applying the functor 
$Hom_{A^e}(-,A)$ 
to the standard bar resolution, 
the Hochschild cochain complex is obtained 
by identifying  
$Hom_{A^e}(A \otimes_k A^{\otimes^n_k} \otimes_k A,A)$ 
to 
$Hom_{k}(A^{\otimes^n_k},A)$. 
The {\it{reduced complex}} is obtained from 
the reduced bar resolution in a similar way.  
First we apply  $Hom_{A^e}(-,A)$ to the reduced bar resolution, 
then we identify the vector space 
$Hom_{A^e}(A \otimes_E r^{\otimes^n_E} \otimes_E A,A)$
to $Hom_{E^e}(r^{\otimes^n_E},A)$. Therefore, 
the reduced bar complex that we
denote by $\bold{R^\bullet(A,A)}$ is given by
\[
\begin{array}{rll}
0  \rightarrow  
A^E \stackrel{\delta}{\rightarrow} 
Hom_{E^e}(r,A) \stackrel{\delta}{\longrightarrow} & \cdots &\, \\
\cdots \: \longrightarrow 
& Hom_{E^e}(r^{\otimes^n_E},A) \stackrel{\delta}{\longrightarrow} 
Hom_{E^e}(r^{\otimes^{n+1}_E},A) & \cdots
\end{array}
\]
where $A^E$ is the subalgebra of $A$ defined as follows:
\[
A^E:= \{ a \in A \, | \,  ae=ea \text{ for all } e \in E\}.
\]
The differentials in the reduced complex are given as the above formulas.

\subsection*{Induced quasi-isomorphism.}
In this paragraph, we will compute the quasi-ismorphisms 
between the Hochschild cochain complex and the
reduced complex, 
induced by the comparison maps $p$ and $s$. We will denote them by  
$$
p^\bullet: \bold{R^\bullet(A,A)} \rightarrow \bold{C^\bullet(A,A)} 
\text{ \, and \, } 
s^\bullet: \bold{C^\bullet(A,A)} \rightarrow \bold{R^\bullet(A,A)}.$$
\begin{map}[$p^\bullet$]
In degree zero, we have that 
$
p_0:A^E \rightarrow A
$ 
is the inclusion map. For $n \geq 1$,  
\[
p^n: Hom_{E^e}(r^{\otimes^n_E},A) \longrightarrow Hom_k(A^{\otimes^n_k},A)
\]
is given by
\[
p^nf (x_1 \tensork \cdots \tensork x_n)=
f(\pi(x_1) \tensorE \cdots \tensorE \pi(x_n))
\]
where 
$f$ is in $Hom_{E^e}(r^{\otimes^n_E},A)$ and $x_i \in r$.
\end{map}
\begin{map}[$s^\bullet$]
In degree zero, we have that 
$
s^0: A  \rightarrow A^E
$
is given by
\[
s^0(x)=\sum_{e \in E_0} exe 
\]
where $x \in A$.
For $n \geq 1$, we have that
\[
s^n: Hom_k(A^{\otimes^n_k},A) \longrightarrow Hom_{E^e}(r^{\otimes^n_E},A)
\]
is given by
\[
s^nf(x_1 \tensorE \cdots \tensorE x_n)= 
\sum_{ j_0, \dots, j_{n}} 
e_{j_0} f(e_{j_0} x_1e_{j_1} \tensork \cdots
\tensork e_{j_{i-1}}x_{i}e_{j_i} \tensork \cdots \tensork
e_{j_{n-1}}x_n e_{j_n}) e_{j_n}
\]
where the sum is over all $(n+1)$-tuples 
$(e_{j_0},\dots,e_{j_i},\dots,e_{j_{n}})$ of elements of $E_0$, 
$f$ is in $Hom_{k}(A^{\otimes^n_k},A)$ and $x_i$ is in $r$.
\end{map}
\begin{remark}
Let us remark that $s^{\bullet} \, p^{\bullet}= id_{\bold{R^\bullet(A,A)}}$.
\end{remark}

%%%%%%%%%%%%%%%%%%%%%%%%%%%%%%%%%%%%%%%%%%%%%%%%%%%%%%%%%%%%%%%%%%%%%%%%%%%%
%%%%%%%%%%%%%%%%% Second Section %%%%%%%%%%%%%%%%%%%%%%%%%%%%%%%%%%%%%%%%%%%
\section{Gerstenhaber bracket and reduced bracket.} 
The Gerstenhaber bracket is defined 
on the Hochschild cohomology groups 
using the Hochschild complex. 
In this section we will define the 
reduced bracket using the 
reduced complex. 
We show that the Gerstenhaber bracket and the reduced bracket
provides the same graded Lie algebra structure on $HH^{*+1}(A)$. 
We begin by recalling 
the Gerstenhaber bracket in order to fix notation.
\subsection*{Gerstenhaber bracket.} 
Set $C^0(A,A):=A$ and for $n \geq 1$, 
we will denote the space of Hochschild cochains by 
\[
C^n(A,A):= Hom_{k}(A^{\otimes^n_k},A). 
\]
In \cite{gersten}, Gerstenhaber defined a right 
pre-Lie system $\{ C^n(A,A), \circ_i\}$ 
where elements of $C^n(A,A)$ are declared to have degree $n-1$. 
The operation $\circ_i$ is given as follows. 
Given $n \geq 1$, let us fix $i = 1,\dots, n$. 
The bilinear map
\[
\circ_i: C^n(A,A) \times C^m(A,A) \longrightarrow C^{n+m-1}(A,A)
\]
is given by the following formula:
\begin{center}
\scalebox{0.97}{$
f^n \circ_i g^m (x_1 \otimes \cdots \otimes x_{n+m-1}):=
f^n(x_1 \otimes \cdots \otimes g^m(x_i \otimes \cdots \otimes x_{i+m-1}) 
\otimes \cdots \otimes x_{n+m-1})$}
\end{center}
where $f^n$ is in $C^n(A,A)$ and $g^m$ is in $C^m(A,A)$. 
Then he proved that such pre-Lie system 
induces a graded pre-Lie algebra structure on 
\[
C^{*+1}(A,A):=\bigoplus_{n=1}^\infty C^n(A,A)
\]
by defining an operation $\circ$ as follows:
\[
f^n \circ g^m:= \sum_{i=1}^n (-1)^{(i-1)(m-1)} f^n \circ_i g^m.
\]
Finally, $C^{*+1}(A,A)$ becomes a 
graded Lie algebra by defining the bracket 
as the graded commutator of $\circ$. So we have that
\[
[f^n \, , \, g^m]:=f^n \circ g^m - (-1)^{(n-1)(m-1)} g^m \circ f^n.
\]  
\begin{remark}
The Gerstenhaber 
restricted to $C^1(A,A)$ is the usual Lie commutator bracket.
\end{remark}
Moreover, Gerstenhaber proved that 
\[
\delta[f^n \, , \, g^m]= 
[f^n \, , \, \delta g^m] + (-1)^{m-1}[\delta f^n \, , \, g^m]
\]
where $\delta$ is the differential of Hochschild cochain complex.
This formula implies 
that the  following bilinear map:
\[
[\, - \, , \, - \,]: HH^n(A) \times HH^m(A) \longrightarrow HH^{n+m-1}(A) 
\]
is well defined.
Therefore, 
$HH^{*+1}(A)$ endowed with the induced Gerstenhaber
bracket is also a graded Lie algebra.

\subsection*{Reduced Bracket.}
In order to define the reduced bracket, 
we proceed in the same way as Gerstenhaber did. 
We will define the reduced bracket as the graded commutator of an
operation $\circR$. Such operation will be given by $\circulito$. 
Denote by $C^n_E(r,A)$ the cochain space of the reduced complex, 
this is 
\[
C^n_E(r,A):= Hom_{E^e}(r^{\otimes^n_E},A).
\]
\begin{definition}
Let $n\geq 1$ and fix $i=1, \dots , n$. 
The bilinear map
\[
\circulito: C^n_E(r,A) \times C^m_E(r,A) \rightarrow C^{n+m-1}_E(r,A)
\]
is given by the following formula:
\[
f^n \circulito g^m(x_1 \tensorE \cdots \tensorE x_{n+m-1}):=
f^n(x_1 \tensorE \cdots \tensorE 
\pi g^m(x_i \tensorE \cdots \tensorE x_{i+m-1}) 
\tensorE \cdots \tensorE x_{n+m-1})
\]
where $f^n$ is in $C^n_E(r,A)$ and $g^m$ is in $C^m_E(r,A)$ and 
$x_1,\dots, x_{n+m-1}$ are in $r$. 
Let us remark that the image of $g^m$ does not necessarily belong to
the radical but the image of $\pi g^m$ clearly does. 
Therefore $f^n \circulito g^m$ is well defined. 
\end{definition}
Then we can define $\circR$ on 
\[
C^{*+1}_E(r,A):= \bigoplus_{n=1}^\infty C^n_E(r,A) 
\]
as above but replacing 
$\circulito$ instead of $\circ_i$. This means that 
\[
f^n \circR g^m := \sum_{i=1}^n (-1)^{(i-1)(m-1)}f^n \circulito g^m
\qquad \: \: \:
\]
Let us remark $\circR$ is a graded operation on 
$C^{*+1}_E(r,A)$ by declaring elements of $C^n_E(r,A)$ to have degree
$n-1$.
\begin{definition}
We call the {\it{reduced bracket}}, denoted by $[\, - \, , \, - \,]_R$, 
to the graded commutator bracket of $\circR$. This is, 
\[
[\, - \, , \, - \,]_R: C^n_E(r,A) \times  C^m_E(r,A)  
\longrightarrow  C^{n+m-1}_E(r,A)
\]
is given by
\[
[f^n \, , \, g^m]_R := f^n \circR g^m - (-1)^{(n-1)(m-1)} g^m \circR f^n.
\]
\end{definition}
The following lemmas will relate the Gerstenhaber bracket
and the reduced bracket. 
\begin{lemma}\label{bracketuno}
We have the following formula:
\[
[f^n \, , \, g^m]_R = s^{n+m-1}[\, p^n f^n \, , \, p^m g^m \,].
\]
\end{lemma}
\begin{proof}
A straightforward verification shows that 
\[
f^n \circulito g^m = s^{n+m-1}(\,p^n f^n \circ_i p^m g^m\,).
\]
Since $s^{n+m-1}$ is a linear application 
we have the formula wanted. 
\end{proof}
\begin{lemma} \label{bracketdos}
We have the following formula:
\[
p^{n+m-1}[\, f^n \, , \, g^m \,]_R= [\, p^n f^n \, , \, p^m g^m\, ]
\]
\end{lemma}
\begin{proof}
Since $p^{n+m-1}$ is a complex morphism, we prove that 
\[
p^{n+m-1} (f^n \circulito g^m) = p^nf^n \circ_i p^m g^m 
\]
by a direct computation.
\end{proof}
We will write 
$p^*$ for the morphism 
\[
p^* : C^{*+1}_E(r,A) \longrightarrow C^{*+1}(A,A) 
\]
induced by $p^\bullet$. 
We have the following proposition due to the above lemmas that relate
both brackets.
\begin{proposition}\label{gradedLiealgebra}
The graded product $[\, - \, , \, - \,]_R$ endows $C^*_E(r,A)$ with the
structure of graded Lie algebra. 
We also have that $p^*$ 
is a morphism of graded Lie algebras. 
\end{proposition}
\begin{proof}
Using the lemma \ref{bracketuno}, it is easy
to see that the reduced bracket satisfies the 
graded antisymmetric property 
as a consequence of 
the fact that the Gerstenhaber bracket satisfies the
same condition.
For the graded Jacobi identity, we proceed in the same way. 
First, let us write a formula that relates both brackets, 
using the lemma \ref{bracketuno} and the lemma \ref{bracketdos}  
we have that
\[
\begin{array}{rcl}
[\,[\, f^n \, , \, g^m \,]_R \, , \, h^l \,]_R & = & 
s^{n+m+p-2}[\, p^{n+m-1}[\, f^n \, , \, g^m \,]_R \, , \, p^lh^l \,] \\
\, & = & 
s^{n+m+p-2}[\, [ \, p^nf^n \, , \, p^mg^m \, ] \, , \, p^lh^l \, ]
\end{array}
\]
Then, using the linearity of $s^{n+m+p-2}$ and the fact that the
Gerstenhaber bracket satisfies the graded Jacobi identity 
we have proved that $[\, - \, , \, - \, ]_R$ satisfies the two conditions
of the definition of graded Lie algebra. 
Finally,  
$p^*$ becomes a Lie graded morphism because of lemma \ref{bracketdos}.  
\end{proof}
Now, the reduced bracket induce a bracket in Hochschild cohomology
groups because of the following lemma.
\begin{lemma}\label{diferencial}
Let $\delta$ be the differential of the Hochschild cocomplex 
then we have 
\[
\delta [\, f^n \, , \, g^m \, ]_R= 
[\, f^n \, , \, \delta g^m \,]_R+(-1)^{m-1}[\,\delta f^n \, , \, g^m \,]_R .
\]
Hence we have a well defined bracket in the Hochschild cohomology groups:
\[
[\, - \, , \, - \,]_R: HH^n(A) \times HH^m(A) \longrightarrow
HH^{n+m-1}(A) \, .
\]
\end{lemma}
\begin{proof}
We have that
\[
\begin{array}{rl}
\delta[\, f^n \, , \, g^m \, ]_R  & = 
\delta s^{n+m-1}[\, p^n f^n \, , \, p^m g^m \,]  \\ 
\,              & = s^{n+m-1} \delta [\, p^n f^n \, , \, p^m g^m \, ] \\
\,              & = s^{n+m-1} [\, p^n f^n \, , \, \delta p^m g^m \,] 
          + (-1)^{m-1}s^{n+m-1} [\, \delta p^n f^n \, , \,  p^m g^m \, ] \\
\,              & = s^{n+m-1} [\, p^n f^n \, , \, p^m \delta g^m \, ] 
         + (-1)^{m-1}s^{n+m-1} [\, p^n \delta f^n \, , \,  p^m g^m \, ] \\
\,                 &= [\, f^n \, , \,  \delta g^m \, ]_R + 
               (-1)^{m-1}[\, \delta f^n \, , \, g^m \, ]_R 

\end{array}
\]\end{proof}
We have equipped $HH^{*+1}(A)$ with a graded Lie algebra
structure induced by the reduced bracket. 
We know that $HH^{*+1}(A)$ is already a graded Lie algebra and this
structure is given by the Gerstenhaber bracket. 
We have then the following proposition.
\begin{proposition}
The graded Lie algebra $HH^{*+1}(A)$ endowed with the Gerstenhaber
bracket is isomorphic to $HH^{*+1}(A)$ endowed with the reduced bracket.
\end{proposition}
\begin{proof}
By abuse of notation we continue to write $\overline{p^*}$ for the
automorphism of $HH^{*+1}(A)$ given by the family of morphisms 
$(\overline{p^n})$. Thus, a direct consequence of the above
proposition is that $\overline{p^*}$ becomes an isomorphism of
graded Lie algebras.
\end{proof}

%%%%%%%%%%%%%%%%%%%%%%%%%%%%%%%%%%%%%%%%%%%%%%%%%%%%%%%%%%%%%%%%%%%%%%%%%%%%
%%%%%%%%%%%%%%%% Third Section %%%%%%%%%%%%%%%%%%%%%%%%%%%%%%%%%%%%%%%%%%%%%
\section{Reduced bracket for monomial algebras 
\\ with radical square zero.}
Let $Q$ be a quiver. 
The path algebra $kQ$ is the 
$k$-linear span of the set of paths of $Q$ 
where multiplication is provided by concatenation or zero.
We denote by $Q_0$ the set of vertices 
and $Q_1$ the set of arrows. 
The trivial paths are denoted 
by $e_i$ where $i$ is a vertex.
The set of all paths of length $n$ is denoted by $Q_n$. 

In the sequel, let $A$ be a monomial algebra with 
radical square zero, this is 
\[
A:=\frac{kQ}{ < Q_2 >}.
\] 
The Jacobson radical of $A$ is given by $r=kQ_1$. 
Moreover, the Wedderburn-Malcev decomposition 
of these algebras 
is $A= kQ_0 \oplus kQ_1$ where $E=kQ_0$. 
In this section we are going to describe  
the reduced bracket on $HH^{*+1}(A)$. 
Such bracket is given 
in terms of the combinatorics of the quiver. 
We will use computations of the Hochschild cohomology 
groups of these algebras given by Cibils in
\cite{cibils}.

\subsection*{The reduced complex.}
Notice that in 
the case of monomial algebras with 
radical square zero, 
the middle-sum terms of the coboundary
morphism of the reduced projective resolution 
$\bold{R}$ vanishes because the multiplication of two arrows
is always zero.
Therefore, we have that the coboundary morphism 
is given by the following formula: 
\[
\begin{array}{cl}
\delta (a \otimes x_1 \otimes \dots \otimes x_{n+1} \otimes b) & =
ax_1 \otimes x_2 \otimes \dots \otimes x_{n+1} \otimes b  \\
\, & + \, (-1) ^{n+1} \,
a \otimes x_1 \otimes \dots \otimes x_{n} \otimes x_{n+1}b.
\end{array}
\] 

In \cite{cibils} an isomorphic complex 
to $\bold{R^\bullet(A,A)}$ is given. 
This new complex is obtained in terms of the combinatorics of the
quiver. To describe it we will need to introduce some notation. 
We say that two paths $\alpha$ and $\beta$ 
are {\it{parallels}} 
if and only if they have the same source and the same end. 
If $\alpha$ and $\beta$ are parallel paths 
we  write $\alpha \parallel \beta$. 
Let $X$ and $Y$ be sets consisting of paths of $Q$, 
the set of parallel paths $X \parallel Y$ is given by :
\[
X \parallel Y:\, =\{ \, (\gamma, \gamma ') \,  \in  \, X \times Y 
\mid \, \gamma \parallel \gamma' \, \}.
\]
For example:
\begin{itemize}  
\item $Q_n \parallel Q_0$ 
is the set of {\it{pointed oriented cycles}}, 
this is the set of pairs $(\gamma^n,e)$ 
where $\gamma^n$ is an oriented cycle of length $n$. 
\item $Q_n \parallel Q_1$ is the set of  pairs $(\gamma^n,a)$ 
where the arrow $a$ is a {\it{shortcut}} of the path $\gamma^n$ of
length $n$.  
\end{itemize}
We denote by $k(X \parallel Y)$ the $k$-vector space generated by
the set $X \parallel Y$. 

For each natural number $n$, Cibils defines 
\[
D_n:k(Q_n \parallel Q_0) \rightarrow k(Q_{n+1} \parallel Q_1) 
\]
as follows:
\begin{equation}\label{Dene}
D_n (\gamma^n,e)=
\sum_{a \in Q_1e} (a\gamma^n, a) + (-1)^{n+1}\sum_{a \in eQ_1}
(\gamma^na, a) 
\end{equation}
where the path $\gamma^n$ is parallel to the vertex $e$.

In \cite{cibils}, the Hochschild cohomology groups of a radical square
zero algebra are obtained from the following complex, 
denoted by $C^\bullet(Q)$ :
\[
0 \rightarrow 
k(Q_0 \parallel Q_0) \oplus k(Q_0 \parallel Q_1)
\stackrel{\matrizDcero}{\longrightarrow}
k(Q_1 \parallel Q_0) \oplus k(Q_1 \parallel Q_1)
\stackrel{\matrizDuno}{\longrightarrow}
\cdots 
\quad
\]
\[
\cdots
\:
k(Q_n \parallel Q_0) \oplus k(Q_n \parallel Q_1)
\stackrel{\matrizDn}{\longrightarrow}
k(Q_{n+1} \parallel Q_0) \oplus k(Q_{n+1} \parallel Q_1).
\]
Cibils proved that $C^\bullet(Q)$
is isomorphic to the reduced 
complex $\bold{R^\bullet(A,A)}$ using the following lemma.
\begin{lemma}[\cite{cibils}]\label{cnera}
Let $A:=kQ / <Q_2>$ where $Q$ is a finite quiver. 
The vector space $C^n_E(r,A)=Hom_{E^e}(r^{\otimes^n_E},A)$ 
is isomorphic to 
\[
k(Q_n \parallel Q_0 \cup Q_1)=
k(Q_n \parallel Q_0) \oplus k(Q_n \parallel Q_1).
\]
\end{lemma}

\subsection*{The reduced bracket.}
Once we have the combinatorial description of $C^n_E(r,A)$, we are
going to compute the reduced bracket in the same terms. To do so we
use the above lemma.  
We begin by introducing some notation. 
\begin{notation}
Given two paths: $\alpha^n$ in $Q_n$ and $\beta^m$ in $Q_m$, 
we will suppose that
\[
\begin{array}{rcl}
\alpha^n & = & a_1 a_2 \dots a_n \\
\beta^m  & = & b_1 b_2 \dots b_m
\end{array}
\]
where $a_i$ and $b_j$ are in $Q_1$. 
Under this assumption, we say 
that $a_i$ and $b_j$ are {\it{arrows in the decomposition}} 
of $\alpha^n$ and $\beta^m$, respectively.
Let $i=1, \dots , n$,
if $a_i \parallel \beta^m$, 
we denote by $\alpha^n \diamantei \beta^m$ 
the path in $Q_{n+m-1}$ obtained 
by replacing the arrow $a_i$ with the path $\beta^m$. 
This means
\[
\alpha^n \diamantei \beta^m 
:= a_1 \cdots a_{i-1} b_1 \cdots b_m a_{i+1} \cdots a_n 
\]
If $a_i$ is not parallel to $\beta^m$ then  
$\alpha^n \diamantei \beta^m$ has no sense. 
Clearly,  $\diamantei$ is not commutative. 
For example, let $a$ in $Q_1$. If 
$a \parallel \beta^m$ then we have that 
\[
a \diamanteuno \beta^m = \beta^m 
\]
Now, if $b_i \parallel a$ we have that 
\[
\beta^m \diamantei a = b_1 \dots b_{i-1}  a b_{i+1} \dots b_{m} .
\]
\end{notation}
\begin{definition}
Let $Q$ be a finite quiver and $n\geq 1$. 
Fix $i=1, \dots , n$. 
The bilinear map 
\[
\circulito: 
k(Q_n \parallel Q_0 \cup Q_1) 
\times 
k(Q_m \parallel Q_0 \cup Q_1)   
\longrightarrow 
k(Q_{n+m-1} \parallel Q_0 \cup Q_1)
\]
is given by 
\[
(\alpha^n,x) \circulito (\beta^m,y)=
\delta_{a_i,y} \, \cdot \, (\alpha^n \diamantei \beta^m,x )
\]
where 
\[
\delta_{a_i,y}=\begin{cases}
               1 & \text{ if } a_i=y \\
               0 & \text{ otherwise }
               \end{cases}
\]
and $\alpha^n=a_1 \cdots a_i \cdots a_n.$
\end{definition}

Denote by $C^{*+1}(Q)$ the following vector space
\[
C^{*+1}(Q):= \bigoplus_{n=1}^\infty 
k(Q_n \parallel Q_0) \oplus k(Q_n \parallel Q_1) \quad.
\]

\begin{definition}
Let $Q$ be a finite quiver. 
The biliner map 
\[
[\,-\,,\,-\,]_Q:
k(Q_n \parallel Q_0 \cup Q_1 ) 
\times 
k(Q_m \parallel Q_0 \cup Q_1)   
\longrightarrow 
k(Q_{n+m-1} \parallel Q_0 \cup Q_1)
\]
is defined as follows
\[
\begin{array}{rl}
[\, (\alpha^n,x) \, , \, (\beta^m,y)\,]_Q &
= \displaystyle{\sum_{i=1}^n (-1)^{(i-1)(m-1)}} 
(\alpha^n,x) \circulito (\beta^m,y) \\
\, & -(-1)^{(n-1)(m-1)}
\displaystyle{\sum_{i=1}^m (-1)^{(i-1)(n-1)}} 
(\beta^m,y)\circulito (\alpha^n,x).
\end{array}
\]
\end{definition}

\begin{theorem}
Let $Q$ be a finite quiver. 
The vector space 
$C^{*+1}(Q)$ together with 
the bracket $[\,-\,,\,-\,]_Q$ 
is a graded Lie algebra. 
Moreover, 
if $A:=kQ / <Q_2>$ then 
the graded Lie algebra 
$C^{*+1}_E(r,A)$ endowed with 
the reduced bracket is 
isomorphic to 
$C^{*+1}(Q)$ endowed with 
the bracket $[\,-\,,\,-\,]_Q$.
\end{theorem}
\begin{proof}
Let $Q$ be a finite quiver and $A:=kQ / <Q_2>$. 
Let us remark that  
$C^{*+1}(Q)$ is isomorphic as a vector space 
to $C^{*+1}(r,A)$ because of lemma \ref{cnera}. 
Using the same isomorphism defined by Cibils to prove 
lemma (\ref{cnera}), a straightfoward verification shows 
that the bracket $[\,-\,,\,-\,]_Q$ is
the combinatorial translation of the reduced bracket. 
\end{proof}

\begin{corollary}\label{resultado}
Let $A:=kQ / <Q_2>$ where $Q$ is a finite quiver. 
The graded Lie algebra structure on 
$HH^{*+1}(A)$ given by the Gerstenhaber bracket 
is induced by the graded Lie algebra structure on 
$C^{*+1}(Q)$ given by $[\,-\,,\,-\,]_Q$.
\end{corollary}

%%%%%%%%%%%%%%%%%%%%%%%%%%%%%%%%%%%%%%%%%%%%%%%%%%%%%%%%%%%%%%%%%%%%%%%%%%%%
%%%%%%%%%%%%%%% Fourth Section %%%%%%%%%%%%%%%%%%%%%%%%%%%%%%%%%%%%%%%%%%%%%
\section{Lie module structure of $HH^n(A)$ 
over $HH^1(A)$.}
In this section, we are going to study the Lie module structure of 
$HH^n(A)$ over $HH^1(A)$ when $A:=kQ/Q_2$ in two cases. The first case
is when $Q$ is a loop and the second case is when $Q$ is a two loops quiver. 
\subsection*{The one loop case.}
It is shown in \cite{cibils} that if $char \, k =0$ and $Q$ is the one
loop quiver then  
the function $D_n$, given by the equation (\ref{Dene}), 
is zero when $n$ is  even  and 
$D_{n}$ is injective when $n$ is odd. 
In fact we have the following proposition:
\begin{propo}[\cite{cibils}]
Assume that 
$Q$ is the one loop quiver. 
Let $k$ be a field of characteritic zero 
and $A:=kQ/<Q_2>$. 
Then we have that $HH^{0}(A) \cong  A$
and for $n>0$ we have that 
\[
HH^n(A) \cong
\begin{cases}
\displaystyle k (Q_n \parallel Q_0)  
& \text{if n
  is even} \\
\, & \,\\
 \displaystyle k (Q_n \parallel Q_1) & 
\text{if n is odd \qquad \qquad } \\
\end{cases}
\]
Therefore, for $n \geq 0$ 
the Hochschild cohomology group $HH^n(A)$ is one dimensional.
\end{propo}
\begin{proposition}
Assume that $Q$ is the one loop quiver,  
where $e$ is the vertex and $a$ is the loop.
Let $k$ be a field of characteritic zero 
and $A:=kQ / <Q_2>$. 
Then 
$HH^1(A)$ is the one dimensional (abelian) Lie algebra and 
the Lie module structure on the Hochschild cohomology groups given by
the Gerstenhaber bracket 
\[
HH^1(A) \times HH^n(A) \longrightarrow HH^n(A)
\]
is induced by the following morphisms:

If $n$ is even, we have that 
\[
\scalebox{0.97}{$
k(Q_1 \parallel Q_1) \times k(Q_n \parallel Q_0) 
\longrightarrow k(Q_n\parallel Q_0)
$}
\]
is given as follows 
\[
(a,a).(a^n,e) = - \, n \, (a^n,e).
\]
If $n$ is odd, we have that 
\[
\scalebox{0.97}{$
k(Q_1 \parallel Q_1) \times k(Q_n \parallel Q_1) 
\longrightarrow k(Q_n\parallel Q_1)
$}
\]
is given as follows 
\[
(a,a).(a^n,a) = - \, (n-1) \, (a^n,a).
\]
So, the Lie module $HH^n(A)$ over $HH^1(A)$ corresponds to the 
one dimensional standard module over $k$. 
\end{proposition}
\begin{proof}
It is an immediate consequence of 
the definition of the bracket $[\, - \, , \, - \, ]_Q$ 
and the corollary \ref{resultado}. 
\end{proof}
Moreover, we have that
\begin{proposition} 
Let $k$ be a field of characteritic zero, 
$Q$ the one loop quiver and $A:=kQ/<Q_2>$. 
The Lie algebra $HH^{odd}$ is the infinite dimensional Witt algebra. 
\end{proposition}
\begin{proof}
If $n$ and $m$ are odd then, using the formula for the bracket, we have  
$$[\, (a^n,a) \, , \, (a^m,a) \, ]_Q= \, (n-m)  \, (a^{n+m-1},a) \, .$$
\end{proof}

\subsection*{The two loops case.}

In \cite{cibils}, Cibils proved that the function 
$D_n$, given by the equation (\ref{Dene}),
is injective for $n \geq 1$
when $Q$ is neither a loop nor an oriented cycle.
Hence we have the following result:
\begin{teo}[\cite{cibils}]
Let $A:=kQ/<Q_2>$ where $Q$ is the two loops quiver. 
Then, $HH^0(A)=A$
and for $n \geq 1$
$$HH^n(A) \cong 
\frac{k(Q_n \parallel Q_1)}{Im\, D_{n-1}}$$
where 
\[
D_{n-1}:k(Q_{n-1} \parallel Q_0) \longrightarrow k(Q_n \parallel Q_1) 
\]
is given by the formula (\ref{Dene}). 
Moreover, we have that for $n > 1$, 
$$
dim_k HH^n(A)= 2^{n+1}-2^{n-1} \, .
$$ 
\end{teo}
\begin{theorem}\label{formula}
Let $A:=kQ / <Q_2>$ where $Q$ is a finite quiver. 
If $Q$ is not an oriented cycle then 
the Lie module structure on the Hochschild cohomology groups given by
the Gerstenhaber bracket 
\[
HH^1(A) \times HH^n(A) \longrightarrow HH^n(A)
\]
is induced by the following bilinear map:
\[
\scalebox{0.97}{$
k(Q_1 \parallel Q_1) \times k(Q_n \parallel Q_1) 
\longrightarrow k(Q_n\parallel Q_1)
$}
\]
given as follows 
\[
\displaystyle{
(a,x).(\alpha^n,y) 
= \delta_{y,a} \cdot (\alpha^n,a) - 
\sum_{i=1}^n \delta_{x,a_i} \cdot (\alpha^n \underset{i}{\diamond} x,y)
}
\]
where $a \parallel x$ and $y$ is a shortcut of the path $\alpha^n$
whose decomposition into arrows is given by  
$\alpha^n=a_1 \cdots a_i \cdots a_n$. 
The path $\alpha^n \underset{i}{\diamond} x$ is obtained 
by replacing $a_i$ with $x$ if $a_i = y$
\end{theorem}
\begin{proof}
It is an immediate consequence of 
the definition of the bracket $[\, - \, , \, - \, ]_Q$ 
and the corollary \ref{resultado}. 
\end{proof}
In \cite{strametz}, Strametz 
studies the Lie algebra structure on the first
Hochschild cohomology group for monomial algebras. 
She formulates the Lie bracket on $HH^1(A)$ 
using the combinatorics of the quiver. 
Let us remark that the formula given by the above theorem  
gives the Lie bracket on $HH^1(A)$ when we set $n=1$. 
Such formula coincides with the one given in \cite{strametz}. 
Let us describe the Lie algebra $HH^1(A)$.

\begin{proposition}
Assume that $Q$ is the two loops quiver where $e$ is
the vertex and the loops are denoted by $a$ and $b$. 
Let $A:=\mathbb{C}Q/<Q_2>$ where 
$\mathbb{C}$ is the complex number field. 
Then the elements 
\[
\begin{array}{rcl}
H&:=&(b,b)-(a,a) \\
E&:=&(a,b)       \\
F&:=&(b,a) 
\end{array}
\] 
generate a copy of the Lie algebra $sl_2(\mathbb{C})$ 
in $HH^1(A)$.
Moreover, the Lie algebra $HH^1(A)$ is isomorphic to 
$sl_2(\mathbb{C}) \times \mathbb{C}$.
\end{proposition}
\begin{proof}
First notice that $HH^1(A) \cong k(Q_1 \parallel Q_1)$ and that  
the elements $H$, $E$, $F$ and  $I:=(a,a)+(b,b)$ form a basis 
of $HH^1(A)$. A straightforward verification 
of the following relations:
\[
[\, H \, ,  \, E \,]_Q    = 2E ,  \quad 
[\, H \, ,  \, F \,]_Q   = -2F ,  \quad 
[\, E \, ,  \, F \,]_Q    = H 
\]
proves that $HH^1(A)$ contains a copy of $sl_2\mathbb{C}$. 
Finally, it is easy to see that 
\[
[\, I \, ,  \, H \,]_Q    = 0 ,  \quad 
[\, I \, ,  \, E \,]_Q    = 0 ,  \quad 
[\, I \, ,  \, F \,]_Q    = 0 ,
\] 
\end{proof}

In order to study the Lie module $HH^n(A)$ over $HH^1(A)$, 
we will study $HH^n(A)$ as a $sl_2(\mathbb{C})$-module. 
Now, let us recall two
classical Lie theory results, see \cite{erdmann,fulton} for more detail.
\begin{enumerate}
\item Every (finite dimensional) $sl_2 \mathbb{C}$-module has a
  decomposition into direct sum of irreducible modules
\item Classification of irreducible $sl_2 \mathbb{C}$-modules: 
there exists an unique irreducible module for each dimension. We
denote by $V(t)$ the irreducible $sl_2 \mathbb{C}$ module of dimension
$t+1$.
\end{enumerate}
Using the above notation, this means that 
$HH^n(A)$ has a decomposition into direct sum of
irreducible modules over $sl_2 \mathbb{C}$ as follows: 
\[
HH^n(A)= \bigoplus_{t=0}^\infty V(t)^{q_t}
\]
We will determine each $q_t$ and to do
so we will use the usual tools of the classical Lie theory. 
We begin by calculating the eigenvector spaces of $H$ as
endomorphism of $k(Q_n \parallel Q_0)$ and $Im \, D_{n-1}$. 

Given a path $\gamma^n$ in $Q_n$ we denote by 
$a(\gamma^n)$ the number of times that the arrow "$a$" appears in the
decomposition of $\gamma^n$. We also denote by 
$b(\gamma^n)$ the number of times that the arrow "$b$" appears in the
decomposition of $\gamma^n$.
\begin{map}[$v$] 
Define $v$ as the function given by: 
\[
\begin{array}{rrll}
v_n: & Q_n        & \rightarrow & \mathbb{Z} \\
\,     & \gamma^n & \mapsto     & a(\gamma^n)-b(\gamma^n)
\end{array}
\]
\end{map}

\begin{lemma}
For all $\gamma^n$ in $Q_n$ we have that 
\[
\begin{array}{ccc}
H.(\gamma^n,a)&=&(v_n(\gamma^n) - 1) \, (\gamma^n,a) \\
H.(\gamma^n,b)&=&(v_n(\gamma^n) + 1) \, (\gamma^n,b) 
\end{array}
\]
and for all $\gamma^{n-1}$ in $Q_{n-1}$ we have that
\[
\begin{array}{c}
H.D_{n-1}(\gamma^{n-1},e)=v_{n-1}(\gamma^{n-1}) \,
D_{n-1}(\gamma^{n-1},e) \, .
\end{array}
\] 
\end{lemma}

\begin{proof} 
Use the formula given in proposition (\ref{formula}).
\end{proof}

\begin{proposition} 
Assume that $char \, k =0$. 
\begin{enumerate}
\item Consider $H$ as an endomorphism of $k(Q_n \parallel Q_1)$.
The eigenvalues of $H$ are $n+1-2l$ where 
$l=0,\dots n+1 $. 
Denote by $W(\lambda)$ 
the eigenspace of $H$ of the eigenvalue $\lambda$.
We have that $$dim_k W(n+1-2l)=\cnk \, . $$
\item  Consider $H$ as an endomorphism of $Im \, D_{n-1}$
The eigenvalues of $H$ restricted to 
$Im \,  D_{n-1}$ are $n-1-2l$ where 
$l=0,\dots n-1 $. As above, denote by $W(\lambda)$ 
the eigenspace of $H$ of the eigenvalue $\lambda$. 
We have that $$dim_k W(n-1-2l)=\cnnk \, .$$
\end{enumerate}
\end{proposition}
\begin{proof}
$(i)$ From the above lemma, it is clear that the set
\[
\{ (\gamma^n,a) \, \mid \, \gamma^n \in Q_n \} \cup 
\{ (\gamma^n,b) \, \mid \, \gamma^n \in Q_n \}
\] 
is a basis of $k(Q_n \parallel Q_1)$ consisting of eigenvectors.
We also have that $(\gamma^n,a)$ and $(\gamma^n,b)$  
are eigenvectors of eigenvalue 
$v(\gamma^n) + 1$ and $v(\gamma^n) - 1$ respectively. 
Since $a(\gamma^n)+b(\gamma^n)=n$ for all paths $ \gamma^n$, 
we have that 
$v(\gamma^n)=n-2b(\gamma^n)$ where $b(\gamma^n)$ 
varies from $0$ to $n$. 
Then we have that $v(\gamma^n) \pm 1$ is of the form 
$n+1-2l(\gamma^n)$ where $l= 0 \dots, n+1 $.
Let us remark the following:
\begin{enumerate}
\item[$-$]  $(a^n,b)$ is the only eigenvector of value $n+1$
\item[$-$]  $(b^n,a)$ is the only eigenvector of value $-(n+1)$ 
\item[$-$]  If $0 < l < n+1 $, we have that 
\begin{itemize}
\item $(\gamma^n,a)$ is an eigenvector of eigenvalue $n+1-2l$ 
iff $l=b(\gamma^n)$
\item $(\gamma^n,b)$ is an eigenvector of eigenvalue $n+1-2l$ 
iff $l-1=b(\gamma^n)$
\end{itemize}
\end{enumerate}
On the other hand, if $0 < l < n+1$, we know that 
there are \scalebox{0.7}{$\left( \begin{array}{c} n \\ l \end{array}\right)$} 
paths $\gamma^n$ such that $b(\gamma^n)=l$ and 
\scalebox{0.7}{$\left( \begin{array}{c} n \\ l-1 \end{array}\right)$} 
paths $\gamma^n$ such that $b(\gamma^n)=l-1$. 
Therefore, there are 
\[
\scalebox{0.8}{$
\left( \begin{array}{c} n \\ l \end{array}\right) +
\left( \begin{array}{c} n \\ l-1 \end{array}\right) =
\left( \begin{array}{c} n+1 \\ l \end{array}\right)$}
\]
eigenvectors $(\gamma^n,x)$ of eigenvalue $n+1-2l$.\\
$(ii)$ From the above lemma, it is clear that the set 
\[
\{ D_{n-1}(\gamma^{n-1},e) \, \mid \, \gamma^{n-1} \in Q_{n-1} \} 
\] 
is a basis of $Im\, D_{n-1}$ consisting of eigenvectors.
We also have that $D_{n-1}(\gamma^{n-1},e)$ is an eigenvector of
eigenvalue $v(\gamma^{n-1})$. 
Since $a(\gamma^{n-1}) + b(\gamma^{n-1}) = n-1$ for all paths 
$\gamma^{n-1}$, we have that 
$v(\gamma^{n-1})=n-1-2b(\gamma^{n-1})$ where  $b(\gamma^n)$ 
varies from $0$ to $n-1$. Therefore the eignevalues are of the form 
$n-1-2l$ where $l$ varies from $0$ to $n-1$ and 
there are 
\scalebox{0.7}{$\left( \begin{array}{c} n-1 \\ l \end{array}\right)$} 
eigenvectors of eignevalue $n+1-2l$. 
\end{proof}

Recall the following result from Lie theory:
\begin{lemma}[General Multiplicty Formula \cite{bremner}]
Let $V$ a finite dimensional $sl_2\mathbb{C}$-module.  
For every integer $t$, 
let $V_t$ be the eigenspace of $H$ of eigenvalue $n$. 
Then for any nonnegative integer $t$, 
the indecomposable module 
the number of copies of $V(t)$ 
that appear in the decomposition into direct sum 
of indecomposable is 
$dim \, V_t \, - \, dim \, V_{t-2}$.
\end{lemma}

A  consequence of the above lemma is the following result:
\begin{lemma}
Let $\mathbb{C}$ be the field of complex numbers, 
$Q$ the quiver given by two loops and 
$A:=\mathbb{C}Q/<Q_2>$. 
For $n \geq 1$, we denote by $h(n)$ the following: 
\[
h(n):= max \, \{ \, l \,\mid \, n+1-2l \geq 0 \, \} 
\] 
and for $l=0, \dots, h(n)$ we denote by $p(n,l)$ the following:
\[
 \scalebox{0.9}{$
p(n,l):=\begin{cases} \cck  & \text{ if $l = 0$} \\
                      \, & \, \\
                      \pnl  & \text{ if $l \geq 1$}

\end{cases}
$}
\]
Then we have that
\begin{enumerate}
\item the decomposition into direct sum  of 
irreducibles of  $\mathbb{C}(Q_n \parallel Q_1)$ 
as $sl_2(\mathbb{C})$ Lie module is given by  
\[
\displaystyle{\mathbb{C}(Q_n \parallel Q_1) 
\cong \bigoplus_{l=0}^{h(n)} V(n+1-2l)^{p(n+1,l)}\, , }
\]
\item the decomposition into direct sum  of 
irreducibles of  $Im \, D_{n-1}$ as $sl_2(\mathbb{C})$ Lie module 
is given by  
\[
\displaystyle{ Im \, D_{n-1} 
\cong \bigoplus_{l=0}^{h(n)-1} V(n-1-2k)^{p(n-1,l)}  \, .}
\]
\end{enumerate}
\end{lemma}
\begin{proposition}
Let $\mathbb{C}$ be the field of complex numbers, 
$Q$ the quiver given by two loops and 
$A:=\mathbb{C}Q/<Q_2>$. 
For $n\geq 1$ and $l=0, \dots, h(n)$ 
we denote by $q(n,l)$ the following:
\[
\scalebox{0.9}{$
q(n,l):=\begin{cases} \mnz  & \text{ if $l = 0,1$} \\
                      \, & \, \\
                      \qnk  & \text{ if $l \geq 2$}

\end{cases}
$}
\]
Then, the decomposition 
of $HH^n(A)$ into a direct sum of irreducible Lie modules
over $sl_2(\mathbb{C})$ is given by  
\[
\displaystyle{HH^n(A) \cong \bigoplus_{l=0}^{h(n)} V(n+1-2l)^{q(n,l)}  }.
\]
\end{proposition}
\begin{algorithm}
There is an algorithm that give us the decomposition of $HH^n(A)$
described in the above proposition. We will explain it in the next paragraph. 
We use the following table to write such decomposition:
\[
\scalebox{0.9}{$
\begin{array}{c||ccccccccc}
n & 
V(0)&V(1)&V(2)&V(3)&V(4)&V(5)&V(6)&V(7)&\cdots \\
\hline \hline 
\, & 
\,  & \,  & \, & \,  & \, & \, & \, & \, & \, \\
HH^2(A) & 
\,  & 1  & \, & 1  & \, & \, & \, & \, & \, \\ 
\rotatebox{90}{$\cdots$} & 
\,  & \,  & \, & \,  & \, & \, & \, & \,& \, \\
HH^n(A) & 
q_0  & q_1  & q_2 & q_3  & q_4 & q_5 & q_6 & q_7 & \cdots \\ 
\, & 
\,  & \,  & \, & \,  & \, & \, & \, & \,& \, \\
\end{array}
$}
\]
In the above table, 
at the row $HH^n(A)$, 
the number that appears in the column $V(t)$ 
states the number of copies of the irreducibble module 
$V(t)$ that appears in the decomposition
of $HH^n(A)$. We leave a blank space if no $V(t)$ appears in the
decomposition of $HH^n(A)$. 
We fix the first row of the table 
with the decomposition of $HH^2(A)$. 
Now, given the entries of the row $HH^n(A)$, we can fill out 
the coefficients of the next row, this is for $HH^{n+1}(A)$, 
in the following manner:
\begin{enumerate}
\item Add an imaginary column $(-)$ just before the column $V(0)$, 
consisting of zeros. 
\item Write down the coefficients of the next row by using 
  the rule from Pascal's triangle: add the number directly
  above and to the left with the number directly above and to the right. 
\end{enumerate}
\[
\scalebox{0.7}{
$
\xymatrix{
\,  \ar@{-}@<+8ex>[ddd]\ar@{-}@<+8.5ex>[ddd] \ar@{-}@<-3.5ex>[rrrrrrrr] 
\ar@{-}@<-3ex>[rrrrrrrr]     
& (-) & V(0)& V(1)&\cdots & V(t-1) & V(t) & V(t+1) & \cdots \\
HH^n(A)   &  0 \ar@{-}[rd] & q_0 & q_1 \ar@{-}[ld] & \cdots 
& q_{t-1} \ar@{-}[rd] & q_t & q_{t+1} \ar@{-}[ld] & \cdots \\
HH^{n+1}(A)&  0  & q_1 & \cdots & \cdots 
& \cdots & q_{t-1}+q_{t+1} & \cdots & \cdots \\
\, & \, & \, & \, &\, & \, & \, & \, &
}
$}
\]
Let us remark that 
the number of copies of $V(1)$ that appear 
in the decomposition of $HH^{n}(A)$ is equal to the number of copies 
of $V(0)$ that appear in the decomposition of $HH^{n+1}(A)$. 
\end{algorithm}
\begin{lemma}We have that 
\begin{enumerate}
\item If $n$ is even then $q(n,h(n))= q(n+1,h(n+1))$. 
\item If $n \geq 2$ then $q(n,l)+q(n,l+1)=q(n+1,l+1)$.
\end{enumerate}
\end{lemma}
\begin{proof}
For the first equality, we verify by a direct computation for $n=2$
and $n=4$. For $n \geq 6$, we use that if $n$ is even then we have that 
\[
\left( \begin{array}{c} n+1 \\ n/2 \end{array} \right) =
\left( \begin{array}{c} n+1 \\ n/2 + 1 \end{array} \right) .
\]
For the second equality, we verify by a direct computation for $l=0$
and $l=1$. For $l \geq 2$, we use the Pascal triangle's rule:
\[
\left( \begin{array}{c} n \\ l \end{array} \right) +
\left( \begin{array}{c} n \\ l + 1 \end{array} \right) =
\left( \begin{array}{c} n+1 \\ l + 1 \end{array} \right) .
\]
\end{proof}
\begin{remark} The algorithm is justify by the above lemma. 
Moreover, we have that  
\[
\scalebox{0.9}
{$q(n,2)=\left( \begin{array}{c} n-1 \\ 2  \end{array}\right)$}. 
\]
This is the reason why we have a section of the Pascal triangle in the
above table.
\end{remark}

Finally, once we have the decompostion of $HH^n(A)$ into direct sum of
irreducible modules over $sl_2 \mathbb{C}$, we return to study
$HH^n(A)$ as a $HH^1(A)$-module.
\begin{corollary}
We have that 
\[
\displaystyle{HH^n(A) \cong \bigoplus_{l=0}^{h(n)} V(n+1-2l)^{q(n,l)} 
\otimes \mathbb{C}}
\]
as Lie modules over $HH^1(A)$.
\end{corollary}
\begin{proof}
Notice that 
$$I.(\gamma^n,x)=(1-a(\gamma^n)-b(\gamma^n))(\gamma^n,x)=(1-n)(\gamma^n,x).$$
\end{proof}

%%%%%%%%%%%%%%%%%%%%%%%%%%%%%%%%%%%%%%%%%%%%%%%%%%%%%%%%%%%%%%%%%%%%%%%%%%%%
%%%%%%%%%%%%%%% Bibliography %%%%%%%%%%%%%%%%%%%%%%%%%%%%%%%%%%%%%%%%%%%%%%%
\bibliography{bibliographie}
\bibliographystyle{amsalpha}

%%%%%%%%%%%%%%%%%%%%%%%%%%%%%%%%%%%%%%%%%%%%%%%%%%%%%%%%%%%%%%%%%%%%%%%%%%%
\end{document}